\documentclass[12pt]{amsart}

\usepackage{AKstyle}

\newcommand{\vertiii}[1]{{\left\vert\kern-0.25ex\left\vert\kern-0.25ex\left\vert #1
    \right\vert\kern-0.25ex\right\vert\kern-0.25ex\right\vert}}
\theoremstyle{plain}

\usepackage{caption}
\usepackage[labelfont=rm]{subcaption}

\usepackage[pagebackref=true, colorlinks]{hyperref}

\hypersetup{pdffitwindow=true,linkcolor=blue,citecolor=blue,urlcolor=blue,menucolor=blue}

\usepackage{comment}

\begin{document}
\title[Lattice points in shears ]{Lattice points in a circle for generic unimodular shears}
\author{Dubi Kelmer}
\thanks{The author is partially supported by NSF grant DMS-1401747.}
\email{kelmer@bc.edu}
\address{Boston College, Boston, MA}

\subjclass{}%
\keywords{}%

\date{\today}%
\dedicatory{}%
\commby{}%

\begin{abstract}
Given a unimodular lattice $\Lambda\subseteq \R^2$ consider the counting function  $\cN_\Lambda(T)$ counting the number of lattice points of norm less than $T$, and the remainder $\cR_\Lambda(T)=\cN(T)-\pi T^2$. We give an elementary proof that the mean square of  the remainder over the set of all shears of a unimodular lattice is bounded by $O(T\log^2(T))$.
\end{abstract}

 \maketitle
\section{Introduction}
For $\Lambda\subset \R^2$ a unimodular lattice let  $\cN_\Lambda(T)$
denote the counting function counting the number of lattice points in a circle of radius $T$. This function grows asymptotically like the area $\pi T^2$ and we write
$$\cN_\Lambda(T)=\pi T^2+ \cR_\Lambda(T),$$
with $\cR_\Lambda(T)$ the remainder function. 
A simple geometric argument, going back to Gauss, implies that $|\cR_\Lambda(T)|\ll T$. Here and below we will use the notation $X\ll Y$ to mean that there is some constant $C$ so that $X\leq C Y$ (when the implied constant depends on parameters we will denote it in a subscript). 

Over the last century there have been many improvements of this bound  \cite{Sierpinski1906,LittlewoodWalfisz1925,Kolesnik1985,IwaniecMozzochi1988,Huxley1993} with the present record $|\cR_\Lambda(T)|\ll_\epsilon T^{131/208+\epsilon}$ due to Huxley \cite{Huxley2003}. This bound is still some distance away from the conjectured bound of $|\cR_\Lambda(T)|\ll_\epsilon T^{1/2+\epsilon}$ conjectured by Hardy \cite{Hardy1915}, who also showed that the exponent $1/2$ is best possible.

In addition to numerical evidence for Hardy's conjecture there are also probabilistic arguments showing that such a bound holds on average. For example, when averaging over the radius we have bounds of the form   
\begin{equation}\label{e:radius}
\frac{1}{H}\int_{T}^{T+H}|\cR_{\Lambda}(t)|^2dt\ll T,
\end{equation}
for various ranges of $H$   \cite{Nowak1985,Bleher1992,Huxley1995,Nowak2002}. 
Another type of average is over deformations in the full space of lattices (not just unimodular). For example, Hofmann, Iosevich, and Weidinger \cite{HofmannIosevichWeidinger2004}, showed that
\begin{equation}\label{e:deform}
\int_{1/2}^2\int_{1/2}^2|E_a(T)|^2da\ll T
\end{equation}
with
$$E_a(T)=\#\{n,m \in \Z | \frac{n^2}{a_1}+\frac{m^2}{a_2}<T^2\}-\sqrt{a_1a_2}\pi T^2.$$ 
By considering this problem spectrally, as estimating the remainder term in Weyl's law for the Laplace spectrum on a flat torus, Petridis and Toth \cite{PetridisToth2002} showed similar bounds for the mean square for more general families of deformations. More generally, Holmin \cite{Holmin13} gave similar bounds for averages over general compact sets of deformations in the full space of lattices. The sets of deformations in all of these examples include lattices and their dilations and hence, as was pointed out in \cite{HofmannIosevichWeidinger2004},  such bounds essentially follow from mean square bounds in the radius of the form of \eqref{e:radius}.

The purpose of this note is to give a mean square estimates when averaging over compact subsets of the space of unimodular lattices. To be more precise, for $z=x+iy$ let $$\Lambda_{z}=\left\{(m\sqrt{y},\tfrac{mx+n}{\sqrt{y}}): m,n\in \Z\right\},$$
and consider the one dimensional family of deformations given by shears
 $$\{\Lambda_{z+t}|t\in [0,1]\}.$$
We note that $\Lambda_{z}=\Lambda_{z+1}$, so this is indeed the family of all shears of $\Lambda_z$. 
We will give an elementary and simple proof of the following:

\begin{thm}\label{t:main}
For this family we have the uniform mean square estimate
\begin{equation}\label{e:meansquare}
\int_0^1|\cR_{\Lambda_{x+iy}}(T)|^2dx\ll \tfrac{T}{\sqrt{y}}\log^2(\tfrac{T}{\sqrt{y}})+y^{3/2}T.
\end{equation}
Moreover, this estimate is optimal (up to the logarithmic term) in the sense that for $T\in \sqrt{y} \N$ we also have
\begin{equation}\label{e:meansquare}
\int_0^1|\cR_{\Lambda_{x+iy}}(T)|^2dx\gg y^{3/2}T.
\end{equation}
\end{thm}

We note that any unimodular lattice can be rotated to $\Lambda_{z}$ for some $|z|\geq 1$, and since the counting function is invariant under rotation we get the following immediate consequence:
\begin{cor}
 For any compact set in the space of unimodular lattices  $K \subset \SL_2(\Z)\bk \SL_2(\R)$ 
$$\int_K |\cR_\Lambda(T)|^2d\mu(\Lambda)\ll_K T\log^2(T),$$
with $\mu$ the probability measure on the space of unimodular lattices coming from Haar measure of $\SL_2(\R)$.
\end{cor}

\begin{rem}
It is interesting to compare this result with the classical result of Randol \cite{Randol1970}, regarding the mean square of the remainder when counting primitive vectors. 
In this case, Randol showed that 
$$\int_{\SL_2(\Z)\bk\SL_2(\R)} |\cN'_\Lambda(T)-\zeta(2)^{-1}\pi T^2|^2d\mu(\Lambda)=\zeta(2)^{-1}\pi T^2(1+o(1))$$
where $\cN'_\Lambda(T)$ denotes the number of primitive vectors in $\Lambda$ of norm less than $T$. We see that the remainder here is of average order $\sim T$, compared to a remainder of average order $O(\sqrt{T}\log(T))$ in our setting. We note that, even though the remainder is smaller (on average) when considering all lattice points, in this case the mean square over the full space of unimodular lattices diverges. This is due to the fact that lattices corresponding to points high in the cusp have very short vectors and hence the area of the circle is no longer a good approximation for the counting function.
\end{rem}

\begin{rem}
The lower bound in Theorem \ref{t:main} implies that
$\int_0^1|\cR_{\Lambda_{x+iy}}(T)|^2dx=\Omega(T)$ (meaning that $\overline{\lim}_{T\to\infty}\tfrac{\int_0^1|\cR_{\Lambda_{x+iy}}(T)|^2dx}{T}>0$).
We note that such $\Omega$-results were proved for more general area preserving deformations by Petridies and Toth \cite{PetridisToth2006} in two and three dimensions. 
\end{rem}
\subsection*{Acknowledgments}
We thank Alex Kontorovich, Zeev Rudnick,  and Andreas Strombergsson for their comments on an earlier versions of this work.
\section{Proofs}
We now turn to the proof of our main result. The first ingredient in the proof is the following explicit formula for the counting function.
For any real number $x$ let $[x]$ denote its integer part,  $\{x\}=x-[x]$ its fractional part and let $s(x)=\frac{1}{2}-\{x\}$ denote the odd sawtooth function. For $z=x+iy$  let
\begin{equation}\label{e:HT}
H_T(z)=2\sum_{0<m< \frac{T}{\sqrt{y}}}\left(s(y\sqrt{\tfrac{T^2}{y}-m^2}+mx)+s(y\sqrt{\tfrac{T^2}{y}-m^2}-mx)\right),
\end{equation} and let 
\begin{equation}\label{e:PT}
P(T)=2\sum_{|m|<T}\sqrt{T^2-m^2}.
\end{equation}
We then have
\begin{lem}\label{l:formula}
For $z=x+iy$
$$\cN_{\Lambda_{z}}(T)=yP(\tfrac T{\sqrt{y}})+H_T(z)+(1-2\{\sqrt{y}T\}).$$
\end{lem}
\begin{proof}
Writing a general element $v\in \Lambda_z$ as $$v=(m\sqrt{y},\tfrac{mx+n}{\sqrt{y}}),\quad m,n\in \Z,$$ we get the identity
\begin{eqnarray*}
\cN_{\Lambda_z}(T)
&=& \#\left\{(m,n\in \Z| (mx+n)^2<yT^2-y^2m^2\right\}\\
&=&\sum_{|m|<\frac{ T}{ \sqrt{y}}}\#\{n||mx+n|<\sqrt{yT^2-y^2m^2}\}\\
&=& \sum_{|m|<\frac{ T}{ \sqrt{y}}} \left([\sqrt{yT^2-y^2m^2}-mx]-[-\sqrt{yT^2-y^2m^2}-mx]\right)\\
&=&\sum_{|m|<\frac{ T}{ \sqrt{y}}}\left(2\sqrt{yT^2-y^2m^2}-\{\sqrt{yT^2-y^2m^2}-mx\}+\{-\sqrt{yT^2-y^2m^2}-mx\}\right)\\
&=& yP(\tfrac T{\sqrt{y}})+\sum_{|m|<\frac{T}{\sqrt{y}}}\left(1-\{y\sqrt{\tfrac{T^2}{y}-m^2}-mx\}-\{y\sqrt{\tfrac{T^2}{y}-m^2}+mx\}\right)\\
&=&yP(\tfrac T{\sqrt{y}}) +H_T(z)+(1-2\{\sqrt{y}T\}).
\end{eqnarray*}
\end{proof}
We think of the first term in the formula as approximating the main term and the second approximates the remainder.
To make this more precise we prove the following estimate

\begin{prop}\label{p:PT}
For all $T\geq 1$ we have
$P(T)=\pi T^2+O(\sqrt{T}).$
\end{prop}
\begin{proof}
Let $s(x)=\tfrac{1}{2}-\{x\}$ denote the odd sawtooth function, let $f_T(x)=\sqrt{T^2-x^2}$, and let 
$\cI_T(M)=\int_0^Mf_T'(x)s(x)dx$. For $M=[T]-1$ we estimate $\cI_T(M)$ in two different ways.

On one hand, using integration by parts
\begin{eqnarray*}
\cI_T(M)&=&\sum_{m=0}^{M-1}\int_n^{n+1}f_T'(x)s(x)dx
=\sum_{m=0}^{M-1}\int_0^1 f_T'(x+m)(\tfrac12-x)dx\\
&=& -\frac12\sum_{m=0}^{M-1}(f_T(m+1)+f_T(m))+\int_0^Mf_T(x)dx\\
&=& \frac{\pi T^2}{4}-\sum_{m=1}^M \sqrt{T^2-m^2}-\frac{T}{2}+\sqrt{T^2-M^2}-\int_M^T\sqrt{T^2-x^2}dx.
\end{eqnarray*}
Noting that $\sum_{m=1}^M \sqrt{T^2-m^2}+\frac{T}{2}=\frac{1}{2}\sum_{|m|\leq M} \sqrt{T^2-m^2}$ we get that
$$\frac{\pi T^2}{4}-\frac{1}{2}\sum_{|m|\leq M} \sqrt{T^2-m^2}=\cI_T(M)+\int_M^T\sqrt{T^2-x^2}dx-\sqrt{T^2-M^2}.$$

On the other hand, integrating by parts the other way we get
\begin{eqnarray*}
\cI_T(M)&=&\sum_{m=0}^{M-1}\int_0^1 f_T'(x+m)(\tfrac12-x)dx\\
&=& \sum_{m=0}^{M-1}\int_0^1(-f_T''(x+m)
\frac{x-x^2}{2}dx
\end{eqnarray*}
Since $-f_T''(x+m)$ is positive and $0<\frac{x-x^2}{2}\leq \frac{1}{8}$ we can bound
\begin{eqnarray*}
0\leq \cI_T(M)&\leq &-\frac{1}{8} \sum_{m=0}^{M-1}\int_0^1f_T''(x+m)
dx\\
&=&-\frac{1}{8}\int_0^M f_T''(x)dx
=\frac{M}{8\sqrt{T^2-M^2}}.
\end{eqnarray*}
Combining the two estimates we see that
$$|\pi T^2-2\sum_{|m|\leq M} \sqrt{T^2-m^2}|\leq \frac{M}{2\sqrt{T^2-M^2}}+4\sqrt{T^2-M^2}.$$

For $ M=[T]-1$ we can bound 
$$\sqrt{2T-1}\leq \sqrt{T^2-M^2}\leq 2\sqrt{T-1},$$ 
and 
$$2\sum_{|m|\leq M} \sqrt{T^2-m^2}=P(T)+O(\sqrt{T}),$$
so that indeed
$P(T)=\pi T^2+O(\sqrt{T})$ as claimed.

\end{proof}

This estimate shows that $yP(\frac{T}{\sqrt{y}})=\pi T^2+O(y^{3/4}\sqrt{T})$ and hence $H_T(z)$ is a good approximation for the remainder, in the following sense:
\begin{cor}\label{c:HR}
$\cR_{\Lambda_z}(T)=H_T(z)+O(y^{3/4}\sqrt{T})$.
\end{cor}

The second ingredient is the following elementary mean square estimate for the oscillatory term $H_T(z)$.
\begin{prop}\label{p:meansquare}
For any $y>0$ and $T>1$ we have 
$$\int_0^1 |H_T(x+iy)|^2dx\ll \max(1,\tfrac{T}{\sqrt{y}}\log^2(\tfrac{T}{\sqrt{y}})).$$
\end{prop}
\begin{proof}
Expanding the sawtooth function $s(x)=1/2-\{x\}$ into its Fourier series
$$s(x)=\sum_{n=1}^\infty\frac{\sin(2\pi n x)}{\pi n},$$
 we get that 
$$H_T(z)=\frac{4}{\pi }\sum_{m< \frac{T}{\sqrt{y}}}\sum_{n=1}^\infty \frac{\sin\left(2\pi n y\sqrt{\tfrac{T^2}{y}-m^2}\right)\cos(2\pi mn x)}{n}.$$
Fixing a large parameter $A\geq 1$ (to be determined later) we separate the sum over $n$ into two ranges 
\begin{eqnarray*}H_T(z)&=&\sum_{m< \frac{T}{\sqrt{y}}}\sum_{n<A} \frac{\sin\left(2\pi n y\sqrt{\tfrac{T^2}{y}-m^2}\right)\cos(2\pi mn x)}{n}\\
&&+\sum_{m< \frac{T}{\sqrt{y}}}\sum_{n=A}^\infty \frac{\sin\left(2\pi n y\sqrt{\tfrac{T^2}{y}-m^2}\right)\cos(2\pi mn x)}{n}\\
&&=\cJ_1+\cJ_2,
\end{eqnarray*}
say, so that 
\begin{eqnarray*}
|H_T(z)|^2&\leq& 2(|\cJ_1|^2+|\cJ_2|^2)
\end{eqnarray*}
We can bound each of the terms as follows: For the first term, exchanging the order of summation and using Cauchy-Schwarz on the $n$ sum we get 
\begin{eqnarray*}
|\cJ_1|^2\leq \log(A)\sum_{n\leq A}\frac{1}{n}\bigg|\sum_{m< \frac{T}{\sqrt{y}}}\sin\left(2\pi n y\sqrt{\tfrac{T^2}{y}-m^2}\right)\cos(2\pi mn x)\bigg|^2,
\end{eqnarray*}
and for the second term,  we don't exchange orders and use Cauchy-Schwarz in the $m$ sum to get
\begin{eqnarray*}
|\cJ_2|^2\leq \frac{T}{\sqrt{y}}\sum_{m< \frac{T}{\sqrt{y}}}\bigg|\sum_{n=A}^\infty \frac{1}{n}\sin\left(2\pi n y\sqrt{\tfrac{T^2}{y}-m^2}\right)\cos(2\pi mn x)\bigg|^2
\end{eqnarray*}
Note that, in both cases, the dependence on $x$ is only in the inner most sum giving the bound 
\begin{eqnarray*}
\int_{\R/\Z}|H_T(z)|^2dx&\ll& \log(A)\sum_{n<A}\frac{1}{n}\int_{0}^{1}\bigg|\sum_{m< \frac{T}{\sqrt{y}}}\sin\left(2\pi n y\sqrt{\tfrac{T^2}{y}-m^2}\right)\cos(2\pi mn x)\bigg|^2dx\\
\nonumber &&+ \frac{T}{\sqrt{y}}\bigg|\sum_{m< \frac{T}{\sqrt{y}}}\int_{0}^{1}|\sum_{n=A}^\infty \frac{1}{n}\sin\left(2\pi n y\sqrt{\tfrac{T^2}{y}-m^2}\right)\cos(2\pi mn x)\bigg|^2dx
\end{eqnarray*}
Next, using orthogonality, we can evaluate the inner integrals  by 
\begin{eqnarray*}
\int_{0}^1\bigg|\sum_{m< \frac{T}{\sqrt{y}}}\sin(2\pi n y\sqrt{\tfrac{T^2}{y}-m^2})\cos(2\pi mn x)\bigg|^2dx=
\frac{1}{2}\sum_{m< \frac{T}{\sqrt{y}}} \sin^2(2\pi n y\sqrt{\tfrac{T^2}{y}-m^2}) \ll \frac{T}{\sqrt{y}},
\end{eqnarray*}
and
\begin{eqnarray*}
\int_{0}^1\bigg|\sum_{n=A}^\infty \frac{\sin(2\pi n y\sqrt{\tfrac{T^2}{y}-m^2})\cos(2\pi mn x)}{n}\bigg|^2dx=
\frac{1}{2}\sum_{n=A}^\infty  \frac{\sin^2(2\pi n y\sqrt{\tfrac{T^2}{y}-m^2})}{n^2}\ll \frac{1}{A}.
\end{eqnarray*}
Plugging these back we get that
\begin{eqnarray*}
\int_{\R/\Z}|H_T(x+iy)|^2dx&\ll& \log^2(A)\frac{T}{\sqrt{y}}+ \frac{T^2}{Ay}\end{eqnarray*}
and taking $A=\max(2,\tfrac{T}{\sqrt{y}})$ concludes the proof.
\end{proof}

\begin{proof}[Proof of Theorem \ref{t:main}]
From Corollary \ref{c:HR} we get that for $z=x+iy$
$$|\cR_{\Lambda_z}(T)|^2\ll |H_T(z)|^2+O(y^{3/2}T),$$
and after integrating,  the mean square estimate \eqref{e:meansquare} follows from Proposition \ref{p:meansquare}.

For the lower bound, note that $H_T(z)$ has mean zero so
$$\int_0^1 \cN_{\Lambda_z}(T)  dx=yP(T/\sqrt{y})+O(1).$$
When $\tilde{T}=T/\sqrt{y}\in \N$ is an integer, $P(\tilde T)$ is the area of the polygon with vertices at the points $(\pm m,\pm \sqrt{\tilde T^2-m})$ with $m=0,1,\ldots , \tilde T-1$. Consequently, in this case it approximates the area of the circle from below and $\pi \tilde T^2-P(\tilde T)\geq \sqrt{2\tilde T-1}$. 
Hence, we have on one hand 
$$|\int_0^1 \cR_{\Lambda_z}(T)  dx|\gg y^{3/4}\sqrt{T},$$
and on the other hand, 
$$|\int_0^1 \cR_{\Lambda_z}(T)  dx|\leq \left(\int_0^1 |\cR_{\Lambda_z}(T)|^2  dx\right)^{1/2},$$
so that indeed,
 $$\int_0^1 |\cR_{\Lambda_z}(T)|^2  dx\gg y^{3/2}\sqrt{T}.$$
\end{proof}



\end{document}